\documentclass[11pt]{amsart}
\usepackage[T1]{fontenc}
\usepackage[utf8]{inputenc}
\usepackage{amsmath,amssymb,amsthm}
\usepackage[hidelinks]{hyperref}
\usepackage[a4paper,margin=1.15in]{geometry}

\newtheorem{theorem}{Theorem}[section]
\newtheorem{lemma}[theorem]{Lemma}
\newtheorem{remark}[theorem]{Remark}
\newtheorem{problem}[theorem]{Problem}
\newtheorem*{theoremA}{Theorem A}
\newtheorem*{theoremB}{Theorem B}
\newtheorem*{theoremC}{Theorem C}

\newtheorem*{conjectureD}{Dombi's Conjecture}
\newtheorem*{problemD}{Dombi's Problem}

\theoremstyle{remark}
\newtheorem*{remarkintro}{Remark}
\theoremstyle{plain}

\newcommand{\BaseB}{\mathcal B}
\newcommand{\NewB}{\mathcal C}

\begin{document}

\title[Monotonicity and polynomial growth]{On the Monotonicity of Higher-Fold Representation Functions}
\author{Csaba S\'andor}
\email{sandor.csaba@ttk.bme.hu}
\address{Department of Stochastics, Institute of Mathematics, Budapest University of Technology and Economics, M\H{u}egyetem rkp. 3., H-1111 Budapest, Hungary; \newline \hspace*{4mm}
HUN-REN Alfr\'ed R\'enyi Institute of Mathematics, Re\'altanoda utca 13--15., H-1053 Budapest,  Hungary; \newline \hspace*{4mm}
MTA--HUN-REN RI Lend\"ulet ``Momentum'' Arithmetic Combinatorics Research Group, Re\'altanoda utca 13--15., H-1053 Budapest,  Hungary.}

\author{Quan-Hui Yang}
\email{yangquanhui01@163.com}
\address{School of Mathematical Sciences and Ministry of Education Key Laboratory for NSLSCS,
Nanjing Normal University, Nanjing 210023, P. R. China}
\thanks{The first author is supported by the NKFIH Grants No. K129335, KKP 144059 and the Lend\"ulet ``Momentum'' program of the Hungarian Academy of Sciences (MTA).
The second author is supported by the National Natural Science Foundation of China, Grant No.~12371005.}

\date{}
\begingroup
\renewcommand{\thefootnote}{}
\addtocounter{footnote}{-1}
\endgroup
\begin{abstract}
For a positive integer $h$, let $R_{A,h}(n)$ denote the number of ordered representations
$n=s_1+\cdots+s_h$ with all $s_i\in A$. Let
\[
\BaseB=\{0\}\cup\{m\ge 1:\text{ the base-4 expansion of }m\text{ begins with }1\text{ or }2\}.
\]
Shallit proved that $R_{\BaseB,3}(n)$ is strictly increasing, thereby disproving a 2002 conjecture of Dombi. In this paper, by using linear bounds for
$R_{\BaseB,3}(n+1)-R_{\BaseB,3}(n)$ and a convolution argument, we prove the polynomial order of
$R_{\BaseB,h}(n+1)-R_{\BaseB,h}(n)$ for every integer $h\ge 3$. More precisely, for every integer $h\ge 3$, there exist constants $c_h,C_h>0$, depending only on $h$, such that
\[
 c_h n^{h-2}\le R_{\BaseB,h}(n+1)-R_{\BaseB,h}(n)\le C_h n^{h-2}
\]
for all integers $n\ge 1$. We also construct a co-infinite set $\NewB\subset\mathbb N$ satisfying
$\lim_{n\to\infty}\NewB(n)/n=1$ such that $R_{\NewB,h}(n)$ is strictly increasing for every integer $h\ge 3$. This answers a problem of Dombi posed in 2002. We also pose some problems for further research.
\end{abstract}

\maketitle
\section{Introduction}

Let $\mathbb N=\{0,1,2,\ldots\}$. For a set $A\subseteq \mathbb N$ and an integer $h\ge 1$, define the ordered $h$-fold representation function by
\[
R_{A,h}(n)=\#\{(a_1,\ldots,a_h)\in A^h:\ a_1+\cdots+a_h=n\}.
\]
Equivalently, if
\[
G_A(x)=\sum_{a\in A}x^a,
\]
then
\[
\sum_{n\ge 0}R_{A,h}(n)x^n=G_A(x)^h.
\]
We also write
\[
A(n)=\#\{a\in A:\ a\le n\}
\]
for the counting function of $A$. The problem considered here is when $R_{A,h}(n)$ can be eventually increasing, or strictly increasing, when the complement $\mathbb N\setminus A$ is infinite.

The following theorem, due to Erd\H{o}s, S\'{a}rk\"{o}zy and S\'{o}s (see \cite{ESS4,ESS5}), shows that for the ordered representation function no genuinely co-infinite example is possible.

\begin{theoremA}[Erd\H{o}s--S\'{a}rk\"{o}zy--S\'{o}s]
Let $A\subseteq \mathbb N$. The function $R_{A,2}(n)$ is eventually increasing if and only if $A$ is cofinite; that is, there exists an integer $n_0$ such that
\[
\{n\in\mathbb N:n\ge n_0\}\subseteq A.
\]
\end{theoremA}

Related versions for unordered and restricted unordered two-fold representation functions were also studied by Erd\H{o}s, S\'{a}rk\"{o}zy and S\'{o}s.

For $h\ge 3$, based on the Rudin--Shapiro sequence, Dombi \cite{Dombi} gave a set of density $1/2$ whose higher representation functions are eventually increasing.

\begin{theoremB}[Dombi]
There exists a set $A\subseteq \mathbb N$ of asymptotic density $1/2$ such that, for every fixed integer $h>4$, the function $R_{A,h}(n)$ is eventually increasing.
\end{theoremB}

Dombi \cite{Dombi} also proposed the following conjectural obstruction to strict monotonicity in the co-infinite case.

\begin{conjectureD}
If $\mathbb N\setminus A$ is infinite, then $R_{A,h}(n)$ cannot be eventually strictly increasing.
\end{conjectureD}

This conjecture is now known to be false. Shallit \cite{ShallitDombi} constructed co-infinite sets with strictly increasing three-fold representation functions by methods from automata theory and logic.

\begin{theoremC}[Shallit]
There exists a co-infinite set $A\subseteq \mathbb N$ such that $R_{A,3}(n)$ is eventually strictly increasing.
\end{theoremC}

Let
\[
\BaseB=\{0\}\cup\{m\ge 1:\text{ the base-4 expansion of }m\text{ begins with }1\text{ or }2\}.
\]
For Shallit's result one may take $A=\BaseB$; in fact $R_{\BaseB,3}(n)$ is strictly increasing.
Shallit also remarked that $R_{\BaseB,h}(n)$ is eventually strictly increasing for every fixed integer $h\ge 3$.
In this paper, we improve this result by giving a polynomial order of
$R_{\BaseB,h}(n+1)-R_{\BaseB,h}(n)$ for every integer $h\ge 3$.

\begin{theorem}\label{thm:main}
For every fixed integer $h\ge 3$, there exist constants $c_h,C_h>0$, depending only on $h$, such that for every integer $n\ge 1$,
\[
 c_h n^{h-2}\le R_{\BaseB,h}(n+1)-R_{\BaseB,h}(n)\le C_h n^{h-2}.
\]
Moreover,
\[
 R_{\BaseB,h}(n+1)-R_{\BaseB,h}(n)>0
\]
for every $n\ge 0$.
\end{theorem}

\begin{remarkintro}
The exponent $h-2$ is the natural scale. For the full set $\mathbb N$,
\[
R_{\mathbb N,h}(n)=\binom{n+h-1}{h-1},
\]
and hence
\[
R_{\mathbb N,h}(n+1)-R_{\mathbb N,h}(n)
=\binom{n+h-1}{h-2}\asymp_h n^{h-2}.
\]
Thus Theorem~\ref{thm:main} says that the base-$4$ set $\BaseB$ has the same first-difference order as $\mathbb N$, despite having an infinite complement with long gaps.
\end{remarkintro}

In Dombi's paper, the following problem is also posed.

\begin{problemD} Let $h \geq 3$. Is there any set $A \subset \mathbb{N}$ with $R_{A,h}(n)$ eventually increasing and satisfying

$$
\lim _{n \rightarrow \infty} \frac{A(n)}{n}=\alpha \quad \text { with } \alpha \neq \frac{1}{2} ?
$$
\end{problemD}

For the set $\BaseB$ one has
\[
\liminf_{n\to\infty}\frac{\BaseB(n)}{n}=\frac23,
\qquad
\limsup_{n\to\infty}\frac{\BaseB(n)}{n}=\frac89.
\]
Hence $\lim_{n\to\infty}\BaseB(n)/n$ does not exist.

In the following theorem, we give a new set $\NewB$ such that $\lim_{n\to\infty}\NewB(n)/n$ exists and is not equal to $1/2$, and such that $R_{\NewB,h}(n)$ is strictly increasing for every fixed integer $h\ge 3$. This answers Dombi's problem in the form stated above. The more delicate case of prescribed density $d\in [0,1)$ remains open and is formulated in Problem~\ref{prob:density} below.

\begin{theorem}\label{thm:main2}
There exists a co-infinite set $\NewB \subset \mathbb{N}$ such that
\[
\lim_{n \rightarrow \infty} \frac{\NewB(n)}{n}=1
\]
and, for every integer $h\ge 3$,
\[
R_{\NewB,h}(n+1)>R_{\NewB,h}(n)\qquad(n\ge 0).
\]
\end{theorem}

\begin{remark}
In the above theorem, we construct $\NewB=\mathbb N\setminus \{2^j:\ j\ge 10\}$. It would be interesting to determine whether the set $\mathbb N\setminus \{4^{n+2}:\ n\ge 0\}$ from Bell and Shallit \cite{BellShallit} also has this property.
\end{remark}

By numerical computations, we can see that the set $A$ in Theorem B by Dombi does not satisfy
$R_{A,3}(n)$ strictly increasing, since $R_{A,3}(22)=78$ and $R_{A,3}(23)=75$. Motivated by this,
we pose the following problems for further research.

\begin{problem}\label{prob:density} For which \(d\in[0,1)\) does there exist a set \(A\subseteq\mathbb N\) with
\[
\lim_{n\to\infty}\frac{A(n)}{n}=d
\]
such that \(R_{A,h}(n)\) is strictly increasing for every integer \(h\ge 3\)?
\end{problem}

\begin{problem} Given two real numbers $\alpha$ and $\beta$ with $0\le \alpha<\beta\le 1$, does there exist $A \subseteq \mathbb{N}$ with $\liminf\limits_{n\rightarrow +\infty}\frac{A(n)}{n}=\alpha$ and
 $\limsup\limits_{n\rightarrow +\infty}\frac{A(n)}{n}=\beta$ such that $R_{A, h}(n)$ is strictly increasing for every $h \geq 3$?
\end{problem}

In the other direction, some results further characterize the obstruction caused by sparsity: if a set $A$ is sparse, then the corresponding unordered, restricted unordered two-fold representation functions  and $R_{A,h}(n)$  cannot be monotonically increasing. For these results, one may refer to \cite{Balasubramanian,ESS4,ESS5,Kiss,KissSandorYang,Tang}.

\section{Notation and recurrences}

Throughout the rest of the paper we write
\[
R_h(n)=R_{\BaseB,h}(n),\qquad \Delta_h(n)=R_h(n+1)-R_h(n)\qquad (n\ge 0).
\]
We put $R_h(n)=0$ for $n<0$. Let
\[
c_n=R_1(n),\qquad b_n=R_2(n),\qquad a_n=R_3(n),
\]
with $a_n=b_n=c_n=0$ for $n<0$, and define
\[
\beta_j=b_{j+1}-b_j,\qquad \alpha_j=a_{j+1}-a_j\qquad(j\in\mathbb Z).
\]
Thus
\[
\alpha_{-2}=\beta_{-2}=0,\qquad \alpha_{-1}=\beta_{-1}=1.
\]

The self-similar description of $\BaseB$ gives
\[
\BaseB=\{0\}\cup \bigcup_{j\ge 0}\{4^j,4^j+1,\ldots,3\cdot 4^j-1\}.
\]
Equivalently, its indicator sequence satisfies, for $n\ge 0$,
\begin{equation}\label{eq:c-recurrence}
 c_{4n}=c_{4n+1}=c_{4n+2}=c_n,\qquad c_{4n+3}=c_n-\delta_{n,0},
\end{equation}
where $\delta_{i,j}$ is the Kronecker delta. Let
\[
C(x)=\sum_{n\ge 0}c_nx^n,\qquad P(x)=1+x+x^2+x^3.
\]
Then \eqref{eq:c-recurrence} is equivalent to the functional equation
\begin{equation}\label{eq:C}
C(x)=P(x)C(x^4)-x^3.
\end{equation}
Since the representations are ordered, the generating function of $R_h$ is $C(x)^h$.

\begin{lemma}\label{lem:beta-rec}
For every $n\ge 0$, one has
\[
\beta_{4n}=\beta_{n-1},\quad
\beta_{4n+1}=\beta_{n-1}+\delta_{n,1},\quad
\beta_{4n+2}=\beta_{n-1}-2c_n+2c_{n-1}-\delta_{n,1},\quad
\beta_{4n+3}=\beta_n.
\]
\end{lemma}

\begin{proof}
Squaring \eqref{eq:C} gives
\begin{equation}\label{eq:C2}
C(x)^2=P(x)^2C(x^4)^2-2x^3P(x)C(x^4)+x^6.
\end{equation}
Since
\[
P(x)^2=1+2x+3x^2+4x^3+3x^4+2x^5+x^6,
\]
coefficient extraction from \eqref{eq:C2} gives, for every $n\ge 0$,
\[
\begin{aligned}
b_{4n}&=b_n+3b_{n-1}-2c_{n-1},\qquad
b_{4n+1}=2b_n+2b_{n-1}-2c_{n-1},\\
b_{4n+2}&=3b_n+b_{n-1}-2c_{n-1}+\delta_{n,1},\qquad
b_{4n+3}=4b_n-2c_n.
\end{aligned}
\]
For example, the coefficient of $x^{4n+2}$ in $P(x)^2C(x^4)^2$ is $3b_n+b_{n-1}$. The coefficients in the remaining two terms are $-2c_{n-1}$ and $\delta_{n,1}$, respectively. Subtracting consecutive formulas in the above four equations gives
\[
\beta_{4n}=\beta_{n-1},\quad
\beta_{4n+1}=\beta_{n-1}+\delta_{n,1},\quad
\beta_{4n+2}=\beta_{n-1}-2c_n+2c_{n-1}-\delta_{n,1},\quad
\beta_{4n+3}=\beta_n.
\]
The conventions $a_n=b_n=c_n=0$ for $n<0$ account for the small exceptional terms appearing at $n=1$ and $n=2$.
\end{proof}

\begin{lemma}\label{lem:alpha-rec}
For every $n\ge 0$, one has
\[
\begin{aligned}
\alpha_{4n}&=2\alpha_{n-2}+2\alpha_{n-1}-3\beta_{n-2}-\delta_{n,2},\\
\alpha_{4n+1}&=3\alpha_{n-1}+\alpha_{n-2}-3\beta_{n-2}+3(c_{n-1}-c_{n-2})+\delta_{n,2},\\
\alpha_{4n+2}&=4\alpha_{n-1}-3\beta_{n-1},\\
\alpha_{4n+3}&=\alpha_n+3\alpha_{n-1}-3\beta_{n-1}.
\end{aligned}
\]
\end{lemma}

\begin{proof}
Cubing \eqref{eq:C} gives
\[
C(x)^3=P(x)^3C(x^4)^3-3x^3P(x)^2C(x^4)^2+3x^6P(x)C(x^4)-x^9.
\]
Extracting coefficients of $x^{4n+s}$, $s=0,1,2,3$, we have
\[
\begin{aligned}
a_{4n}&=a_n+12a_{n-1}+3a_{n-2}-6b_{n-1}-6b_{n-2}+3c_{n-2},\\
a_{4n+1}&=3a_n+12a_{n-1}+a_{n-2}-9b_{n-1}-3b_{n-2}+3c_{n-2}-\delta_{n,2},\\
a_{4n+2}&=6a_n+10a_{n-1}-12b_{n-1}+3c_{n-1},\\
a_{4n+3}&=10a_n+6a_{n-1}-3b_n-9b_{n-1}+3c_{n-1}.
\end{aligned}
\]
To illustrate the extraction, the coefficient of $x^{4n}$ in the term $P(x)^3C(x^4)^3$ is $a_n+12a_{n-1}+3a_{n-2}$, while the remaining terms contribute $-6b_{n-1}-6b_{n-2}+3c_{n-2}$.
Subtracting consecutive formulas in the above four equations and rewriting differences as $\alpha_j$ and $\beta_j$, we have
\[
\begin{aligned}
\alpha_{4n}&=2\alpha_{n-2}+2\alpha_{n-1}-3\beta_{n-2}-\delta_{n,2},\\
\alpha_{4n+1}&=3\alpha_{n-1}+\alpha_{n-2}-3\beta_{n-2}+3(c_{n-1}-c_{n-2})+\delta_{n,2},\\
\alpha_{4n+2}&=4\alpha_{n-1}-3\beta_{n-1},\\
\alpha_{4n+3}&=\alpha_n+3\alpha_{n-1}-3\beta_{n-1}.
\end{aligned}
\]
\end{proof}

\section{Bounds for the two-fold and three-fold differences}

We first prove the two-sided bound for $\beta_n$. The upper bound will be used for the lower estimate on $\alpha_n$; the lower bound is the key input for the upper estimate on $\alpha_n$.

\begin{lemma}\label{lem:beta-bound}
For every $n\ge 0$,
\[
-1\le \beta_n\le 2.
\]
\end{lemma}

\begin{proof}
The initial values are
\[
\beta_0=1,\qquad \beta_1=1,\qquad \beta_2=-1,\qquad \beta_3=1.
\]
Since $\beta_{4m+3}=\beta_m$ by Lemma~\ref{lem:beta-rec}, by induction, it follows that
\begin{equation}\label{eq:beta-special}
\beta_{3\cdot 4^t-1}=\beta_2=-1\quad (t\ge 0),\qquad \beta_{4^t-1}=\beta_3=1\quad (t\ge 1).
\end{equation}
We prove $\beta_m\le 2$ by induction on $m$. Write $m=4N+s$, where $s\in\{0,1,2,3\}$. The cases $s=0,1,3$ follow immediately from Lemma~\ref{lem:beta-rec} and the induction hypothesis. If $s=2$, then
\[
\beta_{4N+2}=\beta_{N-1}-2c_N+2c_{N-1}-\delta_{N,1}.
\]
If $c_{N-1}\le c_N$, the extra term is nonpositive and the induction hypothesis gives the result. If $c_{N-1}=1$ and $c_N=0$, then by the structure of $\BaseB$ we must have $N=3\cdot 4^t$ for some $t\ge 0$. Hence \eqref{eq:beta-special} gives
\[
\beta_{4N+2}=\beta_{3\cdot 4^t-1}+2=1.
\]
Thus $\beta_m\le 2$ in all cases.

It remains to prove $\beta_m\ge -1$. Again write $m=4N+s$. The cases $s=0,1,3$ follow immediately from Lemma~\ref{lem:beta-rec} and the induction hypothesis. For $s=2$, use
\[
\beta_{4N+2}=\beta_{N-1}-2c_N+2c_{N-1}-\delta_{N,1}.
\]
If $(c_{N-1},c_N)\ne (0,1)$, then $-2c_N+2c_{N-1}\ge 0$ except for the harmless subtraction of $\delta_{N,1}$; when $N=1$ the right-hand side is $\beta_0-1=0$. Hence the induction hypothesis gives $\beta_{4N+2}\ge -1$. If $(c_{N-1},c_N)=(0,1)$, then $N=4^t$ for some $t\ge 1$. Therefore \eqref{eq:beta-special} gives
\[
\beta_{4N+2}=\beta_{4^t-1}-2=-1.
\]
This proves $\beta_m\ge -1$ for all $m\ge 0$.
\end{proof}

\begin{lemma}\label{lem:alpha-lower}
For every $n\ge 78$,
\[
\alpha_n=R_3(n+1)-R_3(n)\ge \frac n5+4.
\]
Moreover, $\alpha_n>0$ for every $n\ge 0$.
\end{lemma}

\begin{proof}
Put $L(x)=x/5+4$. We claim the following propagation statement: if $N_0\ge 6$ and
\[
\alpha_n\ge L(n)\qquad (N_0-2\le n<4N_0),
\]
then
\[
\alpha_n\ge L(n)\qquad (n\ge N_0-2).
\]
Indeed, assume inductively that the estimate is known for $N_0-2\le n<4N$, where $N\ge N_0$. Using Lemmas~\ref{lem:alpha-rec} and \ref{lem:beta-bound}, specifically $\beta_j\le 2$, and using $c_{N-1}-c_{N-2}\ge -1$, we get
\[
\begin{aligned}
\alpha_{4N}&\ge 2L(N-2)+2L(N-1)-6\ge L(4N),\\
\alpha_{4N+1}&\ge 3L(N-1)+L(N-2)-9\ge L(4N+1),\\
\alpha_{4N+2}&\ge 4L(N-1)-6\ge L(4N+2),\\
\alpha_{4N+3}&\ge L(N)+3L(N-1)-6\ge L(4N+3).
\end{aligned}
\]
This proves the induction step and hence the propagation statement.

It remains only to verify the initial interval. A direct exact finite verification from the definition of $\BaseB$ gives
\[
\min_{78\le n\le 319}(5\alpha_n-n-20)=3.
\]
Taking $N_0=80$ in the propagation statement yields $\alpha_n\ge n/5+4$ for all $n\ge 78$. The same finite computation gives
\[
\min_{0\le n<78}\alpha_n=1,
\]
so $\alpha_n>0$ for $0\le n<78$, while the displayed lower bound gives positivity for $n\ge 78$.
\end{proof}

\begin{lemma}\label{lem:alpha-upper}
For every $n\ge 0$,
\[
\alpha_n=R_3(n+1)-R_3(n)\le 3n+2.
\]
\end{lemma}

\begin{proof}
Let $U(x)=3x+2$. First, direct computation from the definition of $\BaseB$ gives
\[
\alpha_0,\alpha_1,\ldots,\alpha_7=2,3,1,2,3,7,5,6,
\]
so $\alpha_n\le U(n)$ for $0\le n\le 7$.

We now use strong induction. Let $m\ge 8$, and write $m=4N+s$ with $s\in\{0,1,2,3\}$. Then $N\ge 2$, so the indices $N-2,N-1,N$ are nonnegative and smaller than $m$. Assume $\alpha_j\le U(j)$ for all $j<m$. By Lemma~\ref{lem:alpha-rec}, the lower bound $\beta_j\ge -1$ from Lemma~\ref{lem:beta-bound}, and the trivial inequalities $c_{N-1}-c_{N-2}\le 1$ and $\delta_{N,2}\le 1$, we obtain
\[
\begin{aligned}
\alpha_{4N}&\le 2U(N-2)+2U(N-1)+3=12N-7\le 12N+2=U(4N),\\
\alpha_{4N+1}&\le 3U(N-1)+U(N-2)+7=12N\le 12N+5=U(4N+1),\\
\alpha_{4N+2}&\le 4U(N-1)+3=12N-1\le 12N+8=U(4N+2),\\
\alpha_{4N+3}&\le U(N)+3U(N-1)+3=12N+2\le 12N+11=U(4N+3).
\end{aligned}
\]
Thus $\alpha_m\le U(m)$. The induction is complete.
\end{proof}

\section{The convolution lift and the two-sided bound}

We need a simple lower-density estimate for $\BaseB$.

\begin{lemma}\label{lem:density}
For a positive integer $M$, let
\[
\BaseB(M)=\#\{b\in \BaseB:0\le b\le M\}.
\]
Then
\[\BaseB(M)\ge \frac16 M.
\]
\end{lemma}

\begin{proof}
Choose $t\ge 0$ such that $4^t\le M<4^{t+1}$. If $t=0$, the result is immediate. Assume $t\ge 1$. Since $\BaseB$ contains the intervals $[4^j,3\cdot 4^j-1]$ for $0\le j<t$, we have
\[\BaseB(M)\ge 1+\sum_{j=0}^{t-1}2\cdot 4^j=1+\frac{2(4^t-1)}3>\frac23 \cdot4^t=\frac16 \cdot4^{t+1}>\frac16 M.
\]
\end{proof}

\begin{lemma}\label{lem:convolution}
For positive integers $p,q$ and $n\ge 0$,
\[
R_{p+q}(n)=\sum_{u=0}^n R_p(n-u)R_q(u).
\]
Consequently,
\begin{equation}\label{eq:diff-convolution}
\Delta_{p+q}(n)=\sum_{u=0}^n \Delta_p(n-u)R_q(u)+R_p(0)R_q(n+1).
\end{equation}
\end{lemma}

\begin{proof}
For the first identity, group an ordered $(p+q)$-tuple according to the sum $u$ of its last $q$ entries. The first $p$ entries then have sum $n-u$, giving $R_p(n-u)R_q(u)$ choices. Summing over $u$ gives the convolution formula. Applying this formula at $n+1$ and at $n$, and using the convention $R_p(-1)=0$, gives \eqref{eq:diff-convolution}.
\end{proof}

By Lemma \ref{lem:convolution}, if $\Delta _p(n)>0$ for every $n\ge 0$ for some positive integer $p$, then for every $p'>p$ we have $\Delta _{p'}(n)>0$ for every $n\ge 0$. We also use the following elementary upper estimate.

\begin{lemma}\label{lem:upper-elementary}
For every integer $h\ge 1$ and $N\ge 0$,
\[
\sum_{u=0}^N R_h(u)\le \binom{N+h}{h}
\]
and
\[
R_h(N+1)\le \binom{N+h}{h-1}.
\]
\end{lemma}

\begin{proof}
Since $\BaseB\subseteq \mathbb N$, the number $\sum_{u=0}^N R_h(u)$ is at most the number of ordered $h$-tuples of nonnegative integers with sum at most $N$, which is $\binom{N+h}{h}$. Similarly, $R_h(N+1)$ is at most the number of ordered $h$-tuples of nonnegative integers with sum $N+1$, which is $\binom{N+h}{h-1}$.
\end{proof}

\begin{proof}[Proof of Theorem~\ref{thm:main}]
If $h=3$, then \(R_{\BaseB,3}(n+1)-R_{\BaseB,3}(n)=\alpha_n\). Lemmas~\ref{lem:alpha-lower} and \ref{lem:alpha-upper} give the upper bound and the lower bound after adjusting the constants on the finite interval $1\le n<78$. For instance, one may take
\[
c_3=\min\left\{\frac15,\min_{1\le n<78}\frac{\alpha_n}{n}\right\}>0,
\qquad C_3=5.
\]
The strict positivity for $h=3$ is exactly the second assertion of Lemma~\ref{lem:alpha-lower}.

Now fix $h\ge 4$ and put $k=h-3$.
First we prove the lower bound and positivity. By Lemma~\ref{lem:convolution}, applied with $p=3$ and $q=k$, we have
\begin{equation}\label{eq:lower-conv}
\Delta_h(n)\ge \sum_{u=0}^n \alpha_{n-u}R_k(u).
\end{equation}
Since $R_k(0)=1$ and $\alpha_n>0$ for all $n\ge 0$, this implies
\begin{equation}\label{eq:positive}
\Delta_h(n)>0\qquad (n\ge 0).
\end{equation}
Now assume
\[
n\ge N_h:=\max\{156,4k\}.
\]
If $0\le u\le n/2$, then $n-u\ge n/2\ge 78$. Hence Lemma~\ref{lem:alpha-lower} gives
\[
\alpha_{n-u}\ge \frac{n-u}{5}+4\ge \frac n{10}.
\]
Therefore, by \eqref{eq:lower-conv},
\begin{equation}\label{eq:first-lower}
\Delta_h(n)\ge \frac n{10}\sum_{0\le u\le n/2}R_k(u).
\end{equation}
Set
\[
M=\left\lfloor\frac n{2k}\right\rfloor.
\]
Since $n\ge 4k$, we have $M\ge n/(4k)$. Every ordered $k$-tuple in $(\BaseB\cap[0,M])^k$ has sum at most $kM\le n/2$. Therefore, by Lemma~\ref{lem:density},
\[
\sum_{0\le u\le n/2}R_k(u)\ge  \BaseB (M)^k\ge \left(\frac M6\right)^k\ge \left(\frac n{24k}\right)^k.
\]
Combining this with \eqref{eq:first-lower}, we obtain, for all $n\ge N_h$,
\begin{equation}\label{eq:asymp-lower}
\Delta_h(n)\ge \frac1{10}\left(\frac1{24k}\right)^k n^{k+1}=\frac1{10}\left(\frac1{24(h-3)}\right)^{h-3}n^{h-2}.
\end{equation}
To cover the finite interval $1\le n<N_h$, define
\[
c_h=\min\left\{\frac1{10}\left(\frac1{24(h-3)}\right)^{h-3},\ \min_{1\le n<N_h}\frac{\Delta_h(n)}{n^{h-2}}\right\}.
\]
By \eqref{eq:positive}, the second minimum is taken over finitely many positive numbers, hence $c_h>0$. The lower bound follows from the definition of $c_h$ and \eqref{eq:asymp-lower}.

It remains to prove the upper bound. Again use Lemma~\ref{lem:convolution} with $p=3$ and $q=k$. Since $R_3(0)=1$, Lemma~\ref{lem:alpha-upper} gives, for every $n\ge 0$,
\[
\Delta_h(n)=\sum_{u=0}^n \alpha_{n-u}R_k(u)+R_k(n+1)
\le (3n+2)\sum_{u=0}^n R_k(u)+R_k(n+1).
\]
By Lemma~\ref{lem:upper-elementary},
\begin{equation}\label{eq:upper-binomial}
\Delta_h(n)\le (3n+2)\binom{n+k}{k}+\binom{n+k}{k-1}.
\end{equation}
In particular, if $n\ge 1$, then $n+k\le (k+1)n$ and $3n+2\le 5n$, so
\[
\Delta_h(n)\le \left(5\frac{(k+1)^k}{k!}+\frac{(k+1)^{k-1}}{(k-1)!}\right)n^{k+1}.
\]
Thus the desired upper bound holds with
\[
C_h=5\frac{(h-2)^{h-3}}{(h-3)!}+\frac{(h-2)^{h-4}}{(h-4)!}.
\]
The proof is complete.
\end{proof}

\begin{proof}[Proof of Theorem~\ref{thm:main2}]
Let
\[
E=\{2^j:\ j\ge 10\},\qquad \NewB=\mathbb N\setminus E.
\]
Then $\NewB$ is co-infinite. Moreover, if
\[
E(N)=\#\{a\in E:\ a\le N\},
\]
then $E(N)=O(\log N)$, and hence
\[
\NewB(N)=N+1-E(N).
\]
It follows that
\[
\lim_{N\to\infty}\frac{\NewB(N)}{N}=1.
\]

We next prove strict monotonicity. Let
\[
D(x)=\sum_{a\in E}x^a.
\]
The generating function of $\NewB$ is
\[
G_{\NewB}(x)=\sum_{b\in \NewB}x^b=\frac1{1-x}-D(x).
\]
We first prove the result for $h=3$. Put
\[
S_3(n)=R_{\NewB,3}(n).
\]
Since
\[
\sum_{n\ge0}S_3(n)x^n=G_{\NewB}(x)^3,
\]
we have
\begin{equation}\label{eq:newB-diff}
S_3(n+1)-S_3(n)=[x^{n+1}](1-x)G_{\NewB}(x)^3,
\end{equation}
where $[x^{n+1}](1-x)G_{\NewB}(x)^3$ denotes the coefficient of $x^{n+1}$ in $(1-x)G_{\NewB}(x)^3$.

Let $N=n+1$. Using $G_{\NewB}(x)=1/(1-x)-D(x)$, we get
\[
(1-x)G_{\NewB}(x)^3
=\frac1{(1-x)^2}-\frac{3D(x)}{1-x}+3D(x)^2-(1-x)D(x)^3.
\]
It follows from \eqref{eq:newB-diff} that
\[
S_3(n+1)-S_3(n)=N+1-3E(N)+3r_2(N)-r_3(N)+r_3(N-1),
\]
where $r_i(N)$ denotes the ordered number of representations of $N$ as a sum of $i$ elements of $E$.
Since $r_2(N)\ge0$ and $r_3(N-1)\ge0$, we obtain
\[
S_3(n+1)-S_3(n)\ge N+1-3E(N)-r_3(N).
\]
Every representation
\[
N=a_1+a_2+a_3,\qquad a_1,a_2,a_3\in E,
\]
is determined by the ordered pair $(a_1,a_2)$, since then $a_3=N-a_1-a_2$. Thus
\[
r_3(N)\le E(N)^2.
\]
Consequently,
\[
S_3(n+1)-S_3(n)\ge N+1-3E(N)-E(N)^2.
\]
If $N<2^{10}$, then $E(N)=0$, and so $S_3(n+1)-S_3(n)\ge N+1>0$.
Now assume $N\ge 2^{10}$ and put $m=E(N)$. Then $m\ge 1$, and from the definition of $E$ we have $N\ge 2^{m+9}$. Hence
\[
N+1-3E(N)-E(N)^2=N+1-3m-m^2\ge 2^{m+9}+1-3m-m^2.
\]
For every $m\ge1$, one has $m^2+3m<2^{m+9}$. Therefore
\[
S_3(n+1)-S_3(n)>0
\]
for every $n\ge0$.

Finally let $h\ge4$. By the convolution identity,
\[
R_{\NewB,h}(n)=\sum_{u=0}^n R_{\NewB,3}(n-u)R_{\NewB,h-3}(u).
\]
Thus
\[
\begin{aligned}
R_{\NewB,h}(n+1)-R_{\NewB,h}(n)
&=\sum_{u=0}^{n}\bigl(R_{\NewB,3}(n+1-u)-R_{\NewB,3}(n-u)\bigr)R_{\NewB,h-3}(u)\\
&\quad +R_{\NewB,3}(0)R_{\NewB,h-3}(n+1).
\end{aligned}
\]
Since $0\in \NewB$, the term corresponding to $u=0$ in the above sum is strictly positive, and all other terms are nonnegative. Hence
\[
R_{\NewB,h}(n+1)-R_{\NewB,h}(n)>0.
\]
This proves the theorem for every $h\ge3$.
\end{proof}

\end{document}